\documentclass[11pt]{article}

\usepackage[letterpaper,left=1.25in,right=1.25in,top=1in,bottom=1in]{geometry}
\usepackage{amsmath,amssymb,amsthm}
\usepackage{setspace}
\usepackage{microtype}
\usepackage{needspace}
\usepackage{hyperref}

\allowdisplaybreaks
\newtheorem{theorem}{Theorem}[section]
\newtheorem{conjecture}[theorem]{Conjecture}
\newtheorem{lemma}[theorem]{Lemma}
\newtheorem{claim}{Claim}[section]
\newtheorem*{hittingrestatement}{Theorem~\ref{thm:hitting}}
\newtheorem*{mainrestatement}{Theorem~\ref{thm:main}}
\numberwithin{equation}{section}

\hypersetup{
  hidelinks,
  pdfauthor={Xinheng Lin},
  pdftitle={On separators in K\_h-minor-free graphs},
  pdfsubject={Graph minors and balanced vertex separators},
  pdfkeywords={Alon-Seymour-Thomas conjecture, graph minors, balanced separators,
    connected hitting sets, optimal transport}
}

\title{On separators in $K_h$-minor-free graphs}
\author{Xinheng Lin\\[1.4ex]
  \normalsize Center for Discrete Mathematics, Fuzhou University\\[-0.1ex]
  \normalsize Fuzhou, P.\ R.\ China}
\date{November 23, 2026}

\begin{document}

\setlength{\abovedisplayskip}{6pt plus 2pt minus 2pt}
\setlength{\belowdisplayskip}{6pt plus 2pt minus 2pt}
\setlength{\abovedisplayshortskip}{0pt plus 2pt}
\setlength{\belowdisplayshortskip}{3pt plus 2pt minus 1pt}

\maketitle

\begin{abstract}
In 1990, Alon, Seymour, and Thomas~\cite{alon-seymour-thomas} conjectured that every $n$-vertex
$K_h$-minor-free graph has a balanced separator of size $O(h\sqrt{n})$.
With the help of ChatGPT 6pro, we prove this conjecture.
More precisely, for every integer $h\geq 3$, every $n$-vertex
$K_h$-minor-free graph $G$ has a vertex set $S$ with
$|S|\leq 14h\sqrt{n}$ such that every connected component of $G-S$ has
at most $2n/3$ vertices. Our proof uses an optimal-transport argument
to construct connected sets hitting large vertex sets, which leads to
the desired separator bound.
Conroy, Le, and Spalding-Jamieson~\cite{conroy-le-spalding-jamieson} have
independently obtained a similar result.

\end{abstract}

\footnotetext[1]{
Email: \href{mailto:2500310007@fzu.edu.cn}{2500310007@fzu.edu.cn}}

\section{Introduction}\label{sec:introduction}

All graphs in this paper are finite and simple. For graphs $H$ and $G$,
we write $H\preccurlyeq_{\mathrm m}G$ if $H$ is a minor of $G$. We denote
by $K_h$ the complete graph on $h$ vertices, and say that $G$ is
$K_h$-minor-free if $K_h\not\preccurlyeq_{\mathrm m}G$. A
\emph{balanced separator} of an $n$-vertex graph $G$ is a set
$S\subseteq V(G)$ such that every component of $G-S$ has at most
$2n/3$ vertices.

The study of separator bounds can be traced back to 1979.
Lipton and Tarjan~\cite{lipton-tarjan} proved that every $n$-vertex
planar graph has a balanced separator of size $O(\sqrt{n})$.
In 1984, Gilbert, Hutchinson, and Tarjan~\cite{gilbert-hutchinson-tarjan} showed
that every $n$-vertex graph of genus $g\geq 1$ has a balanced separator
of size $O(\sqrt{gn})$.

In 1990, Alon, Seymour and Thomas~\cite[p.~801]{alon-seymour-thomas}
proved that every $n$-vertex $K_h$-minor-free graph has a balanced
separator of size $O(h^{3/2}\sqrt n)$. Discussing their theorem~(1.2),
they wrote, \textquotedblleft We think that the expression
$h^{3/2}n^{1/2}$ in (1.2) is not the best possible, and that
$O(hn^{1/2})$ is the correct answer \ldots\textquotedblright.
They then observed that, for
suitable positive constants $c$ and $d$, every $K_h$-minor-free
$n$-vertex $3$-regular expander with $h=\lceil c\sqrt{n}\rceil$ requires
at least $dn$ vertices in every balanced separator. Hence the bound
$O(hn^{1/2})$ is asymptotically tight.

We state the unweighted version of their conjecture below.

\Needspace{3.5\baselineskip}
\begin{conjecture}[Alon--Seymour--Thomas~\cite{alon-seymour-thomas}]
\label{conj:ast}
There exists a constant $C>0$ such that, for every integer $h\geq3$,
every $n$-vertex $K_h$-minor-free graph has a balanced separator of size
at most $Ch\sqrt n$.
\end{conjecture}

Plotkin, Rao and Smith~\cite{plotkin-rao-smith} used shallow minors
to obtain balanced separators of size $O(h\sqrt{n\log n})$.
Kawarabayashi and Reed~\cite{kawarabayashi-reed} used the graph minor
structure theorem~\cite{robertson-seymour-structure}. Their approach gives a
size bound of $O(h\sqrt n+f(h))$, where $f(h)$ is an exponential factor.
This shows that Conjecture~\ref{conj:ast} holds for
sufficiently large $n$ in terms of $h$. Kawarabayashi subsequently asked
whether the additive term $f(h)$ could be removed.

Using flow-cut duality, Biswal, Lee and Rao~\cite{biswal-lee-rao}
obtained a size bound of $O(h^3\sqrt{n\log h})$.
Spalding-Jamieson~\cite{spalding-jamieson} obtained
$O(h\log h\log\log h\sqrt n)$ via reweighted spectral partitioning.
Le~\cite[Theorem~1]{le} further reduced the flow-based bound to
$O(h\log h\sqrt n)$.

Bonnet et al.~\cite{bonnet-et-al} gave an algorithm that finds a
balanced separator of size $O(h^{13}\sqrt n)$ in $O(h^{13}n)$ time.
For fixed $h$, Le and Spalding-Jamieson~\cite{le-spalding-jamieson-flow-optimal}
find a balanced separator of size $O(h\log h\sqrt n)$ in expected
linear time.

Le~\cite{le} proved a bound of $O(h\sqrt{n\log h})$ by incorporating a
low-diameter decomposition~\cite{conroy-filtser-padded} into the framework of
Alon, Seymour and Thomas.

We prove Conjecture~\ref{conj:ast} with $C=14$.

\Needspace{5\baselineskip}
\begin{theorem}\label{thm:main}
Let $h\geq3$ be an integer, and let $G$ be a $K_h$-minor-free graph on
$n$ vertices. Then $G$ has a balanced separator $S$ such that
\[
  |S|\leq14h\sqrt n.
\]
\end{theorem}

The proof of Theorem~\ref{thm:main} is based on the original argument of
Alon, Seymour and Thomas~\cite[Theorem~1.2]{alon-seymour-thomas}
and Le's refinement~\cite[Theorem~2]{le}.

For an induced subgraph $H$ of a graph $G$ and a set $W\subseteq V(G)$, let
\[
  N_H(W):=\{v\in V(H)\setminus W:uv\in E(G)\text{ for some }u\in W\}
\]
denote the external neighborhood of $W$ in $H$.
The key to proving Theorem~\ref{thm:main} is to establish the
following theorem on connected hitting sets.

\Needspace{8\baselineskip}
\begin{theorem}\label{thm:hitting}
Let $H$ be a connected graph on $N\geq1$ vertices, let $h\geq3$ be an
integer, and let $\ell\geq1$ be real. Let
$A_1,\ldots,A_h\subseteq V(H)$ satisfy $|A_i|\geq N/h$ for every
$i\in[h]$. Then at least one of the following holds:
\begin{enumerate}
\item there is a nonempty set $W\subseteq V(H)$ such that
$|W|\leq N/2$ and $|N_H(W)|\leq |W|/\ell$;
\item there is a nonempty set $T\subseteq V(H)$ such that $H[T]$ is
connected, $T\cap A_i\neq\varnothing$ for every $i\in[h]$, and
$|T|\leq11h\ell$.
\end{enumerate}
Note that the sets $A_i$ need not be disjoint or connected, and repetitions are
allowed.
\end{theorem}

\Needspace{9\baselineskip}
This result is motivated by the small connecting tree lemma of Alon,
Seymour, and Thomas.
For a graph $G$ and $Z\subseteq V(G)$, a $Z$-flap is the vertex set of
a connected component of $G-Z$.

\begin{theorem}[{Alon--Seymour--Thomas~\cite[(2.1)]{alon-seymour-thomas}}]
\label{thm:ast-connecting-trees}
Let $H$ be a graph with $N$ vertices, let $A_1,\ldots,A_h\subseteq V(H)$,
and let $r\in\mathbf{R}^{+}$ with $r\geq 1$. Then either
\begin{enumerate}
\item[\textup{(i)}] there exists $Z\subseteq V(H)$ with
$|Z|\leq(h-1)N/r$, such that no $Z$-flap intersects all of
$A_1,\ldots,A_h$, or
\item[\textup{(ii)}] there is a tree $T$ in $H$ with $|V(T)|\leq r$ such
that $V(T)\cap A_i\neq\varnothing$ for $i=1,\ldots,h$.
\end{enumerate}
\end{theorem}

A direct application of Theorem~\ref{thm:ast-connecting-trees} yields a
weaker version of Theorem~\ref{thm:hitting}, with $|T|=O(h^2\ell)$ in
the second conclusion: Take $r=4h(h-1)\ell$ in
Theorem~\ref{thm:ast-connecting-trees} and consider every set
$Z\subseteq V(H)$ with $|Z|\leq N/(4h\ell)$. Suppose that the first
conclusion of Theorem~\ref{thm:hitting} fails. Hence any union $W$
of $Z$-flaps with $|W|\leq N/2$ satisfies
$|W|\leq\ell|Z|\leq N/(4h)$, since $N_H(W)\subseteq Z$.
If all flaps had size at most $N/2$, each would have size at most
$N/(4h)$. Since $N-|Z|>N/2$, a smallest union of size greater than
$N/(4h)$ would then have size at most $N/(2h)$ and hence less than
$N/2$, a contradiction.
Thus a flap $C$ has size greater than $N/2$, and the other flaps
have total size at most $\ell|Z|$. Hence
$|V(H)\setminus C|\leq(1+\ell)|Z|<N/h$, so $C$ meets every $A_i$,
contradicting \textup{(i)}. The tree in \textup{(ii)} gives the
required set of at most $4h(h-1)\ell$ vertices.

In fact, a dichotomy of the same form as Theorem~\ref{thm:hitting}
appears in Plotkin, Rao, and Smith~\cite[Lemma~2.5]{plotkin-rao-smith}.
In the tree case, their BFS construction gives $|T|=O(h\ell\log N)$.

Le~\cite[Section~3.1]{le} combines BFS layering
with a low-diameter decomposition derived from the padded decomposition
of Conroy and Filtser~\cite{conroy-filtser-padded}, improving the bound
on $|T|$ to $O(h\ell\log h)$.

When the first conclusion of Theorem~\ref{thm:hitting} fails, we use
optimal transport to construct a low-cost unit flow on a layered graph.
We decompose this flow into paths in the layered graph and obtain a
walk in $H$ that visits at most $11h\ell$ vertices and hits every $A_i$.
This removes the $\log h$ factor in Le's bound on $|T|$.
As in the separator construction of Le~\cite[Section~3.1]{le} (going back
to Plotkin, Rao, and Smith~\cite[Section~2]{plotkin-rao-smith} and Alon,
Seymour and Thomas~\cite{alon-seymour-thomas}), we then shrink a large
component by adding a new branch set, enlarging an existing one, or
cutting off a set with a small boundary.
Theorem~\ref{thm:hitting} guarantees that at least one of these steps
is possible at each stage. At the end of the process, we obtain a
balanced separator $S$ with $|S|\leq14h\sqrt n$.

Section~\ref{sec:preliminaries} bounds the average transport cost.
Section~\ref{sec:hitting} uses this bound to construct a short walk
hitting the target sets. Section~\ref{sec:separator} uses the
resulting connected set in the separator construction.

\clearpage
\section{Preliminaries}\label{sec:preliminaries}

For a graph $G$, write $V(G)$ and $E(G)$ for its vertex and edge sets.
Put $v(G):=|V(G)|$ and $e(G):=|E(G)|$. For $U\subseteq V(G)$, let
$G[U]$ denote the subgraph induced by $U$, and put
$G-U:=G[V(G)\setminus U]$. For disjoint sets $U,W\subseteq V(G)$, let
$e_G(U,W)$ denote the number of edges with one endpoint in $U$ and the
other in $W$. For a nonnegative integer $m$, let
$[m]:=\{1,\ldots,m\}$, with $[0]:=\varnothing$.

We use the standard branch-set characterization of a complete minor.
Thus $K_h\preccurlyeq_{\mathrm m}G$ if and only if there are pairwise
disjoint nonempty sets $C_1,\ldots,C_h\subseteq V(G)$ such that every
$G[C_i]$ is connected and $e_G(C_i,C_j)>0$ whenever $i\neq j$.
The sets $C_i$ are called the branch sets of the minor model.

For a connected graph $H$, let $d_H(x,y)$ denote the number of edges in
a shortest path from $x$ to $y$. For a nonempty set $A\subseteq V(H)$,
put $d_H(x,A):=\min_{a\in A}d_H(x,a)$. A function
$f:V(H)\to\mathbb R$ is $1$-Lipschitz if
$|f(x)-f(y)|\leq d_H(x,y)$ for all $x,y\in V(H)$. Equivalently,
$|f(x)-f(y)|\leq1$ whenever $xy\in E(H)$.

If the first conclusion of Theorem~\ref{thm:hitting} holds, there is
nothing to prove, so for the estimates below we assume it fails.
We call this the expansion assumption.
We can obtain a variance bound of the same order from
Bobkov, Houdr\'e and Tetali~\cite[Theorem~1]{bobkov-houdre-tetali}.
They also improved an earlier concentration bound of
Alon and Milman~\cite[Theorem~2.6]{alon-milman-isoperimetry}.

\Needspace{6\baselineskip}
\begin{lemma}\label{lem:variance}
Let $H$ be a connected graph on $N\geq1$ vertices, and let $\ell\geq1$
be real. Suppose that every nonempty set $W\subseteq V(H)$ with
$|W|\leq N/2$ satisfies $|N_H(W)|>\frac{|W|}{\ell}$.
If $f:V(H)\to\mathbb R$ is $1$-Lipschitz and
$\sum_{v\in V(H)}f(v)=0$, then
$\sum_{v\in V(H)}f(v)^2\leq6N\ell^2$.
\end{lemma}

\begin{proof}
Choose a median $m$ of the values of $f$. Thus at most $N/2$ of these
values are greater than $m$, and at most $N/2$ are less than $m$.
For every integer $j\geq0$, put
$U_j:=\{v\in V(H):f(v)>m+j\}$.
If $u\in U_{j+1}$ and $uv\in E(H)$, then
$f(v)\geq f(u)-1>m+j$. Consequently,
$N_H(U_{j+1})\subseteq U_j\setminus U_{j+1}$.
If $U_{j+1}\neq\varnothing$, then $|U_{j+1}|\leq |U_0|\leq N/2$ and,
by the expansion assumption,
\[
  |U_j|\geq |U_{j+1}|+|N_H(U_{j+1})|
  >\left(1+\frac1\ell\right)|U_{j+1}|.
\]
It follows that
$|U_{j+1}|\leq |U_j|\ell/(\ell+1)$. This weak inequality also holds
when $U_{j+1}=\varnothing$. Induction on $j$ therefore gives
\[
  |\{v:f(v)>m+j\}|
  \leq\frac N2\left(\frac{\ell}{\ell+1}\right)^j
  \qquad(j=0,1,2,\ldots).
\]
Applying the same bound to $-f$, with median $-m$, gives
\[
  |\{v:f(v)<m-j\}|
  \leq\frac N2\left(\frac{\ell}{\ell+1}\right)^j
  \qquad(j=0,1,2,\ldots).
\]
Combining the two bounds gives
\begin{equation}\label{eq:median-tail}
  |\{v:|f(v)-m|>j\}|
  \leq N\left(\frac{\ell}{\ell+1}\right)^j
  \qquad(j=0,1,2,\ldots).
\end{equation}

\Needspace{5\baselineskip}
For $v\in V(H)$, put $r_v:=\lceil|f(v)-m|\rceil$.
Since $r_v$ is a nonnegative integer,
$r_v^2=\sum_{j=0}^{r_v-1}{(2j+1)}$,
where the sum is zero if $r_v=0$.
Also, $r_v>j$ if and only if $|f(v)-m|>j$ for every integer $j\geq0$.
Thus
\begin{align*}
  \sum_v(f(v)-m)^2
  &\leq\sum_v r_v^2\\*
  &=\sum_v\sum_{j=0}^{r_v-1}(2j+1)\\*
  &=\sum_{j=0}^{\infty}\sum_{\{v:r_v>j\}}(2j+1)\\*
  &=\sum_{j=0}^{\infty}(2j+1)|\{v:r_v>j\}|\\*
  &\leq N\sum_{j=0}^{\infty}(2j+1)\left(\frac{\ell}{\ell+1}\right)^j.
\end{align*}
In the last inequality we used~\eqref{eq:median-tail}.
Put $t:=\frac{\ell}{\ell+1}$. Since $0<t<1$, we have
$\sum_{j=0}^{\infty}jt^j=\frac{t}{(1-t)^2}$ and
$\sum_{j=0}^{\infty}t^j=\frac{1}{1-t}$. Hence
\begin{align*}
  N\sum_{j=0}^{\infty}(2j+1)t^j
  &=N\left(\frac{2t}{(1-t)^2}+\frac{1}{1-t}\right)\\*
  &=N(2\ell+1)(\ell+1)\\*
  &\leq6N\ell^2.
\end{align*}
Finally, $\sum_v f(v)=0$ implies
\[
  \sum_v f(v)^2\leq\sum_v f(v)^2+Nm^2
  =\sum_v(f(v)-m)^2\leq6N\ell^2,
\]
as required.
\end{proof}

Let $H$ be a connected graph, and let $a,b:V(H)\to[0,\infty)$ satisfy
$\sum_xa(x)=\sum_xb(x)=1$. A transport matrix from $a$ to $b$ is a
nonnegative matrix $P=(P(x,y))_{x,y\in V(H)}$ such that
\[
  \sum_yP(x,y)=a(x)\quad(x\in V(H)),\qquad
  \sum_xP(x,y)=b(y)\quad(y\in V(H)).
\]
Define
\begin{equation}\label{eq:transport-definition}
  D_H(a,b):=\min_P\sum_{x,y\in V(H)}d_H(x,y)P(x,y),
\end{equation}
where the minimum is over all such matrices.
This set is nonempty, since it contains $P(x,y)=a(x)b(y)$.
It is also a closed subset of $[0,1]^{V(H)\times V(H)}$.
The cost function is continuous and thus attains its minimum on this
compact set.

We use Kantorovich--Rubinstein duality to estimate $D_H(a,b)$.

\Needspace{6\baselineskip}
\begin{lemma}[{Kantorovich--Rubinstein duality,
\cite[Eq.~(6.2)]{peyre-cuturi} and \cite{villani-topics}}]
\label{lem:transport-duality}
Let $H$ be a connected graph, and let $a,b:V(H)\to[0,\infty)$ satisfy
$\sum_xa(x)=\sum_xb(x)=1$. Then
\[
  D_H(a,b)=
  \max_{\substack{f:V(H)\to\mathbb R\\ f\text{ is }1\text{-Lipschitz}}}
  \sum_{x\in V(H)}f(x)(a(x)-b(x)).
\]
\end{lemma}

\phantomsection\label{page:mass-functions}
For nonempty subsets $A_1,\ldots,A_h$ of $V(H)$ and each $i\in[h]$,
we need to distinguish whether a vertex $x$ belongs to $A_i$.
Let $a_i$ be the uniform probability mass function on $A_i$, defined by
\[
  a_i(x):=
  \begin{cases}
    1/|A_i|,&x\in A_i,\\
    0,&x\notin A_i.
  \end{cases}
\]
For $S\subseteq[h]$, write $S^c:=[h]\setminus S$.
If $S\neq\varnothing$, put
$a_S(x):=\bigl(\sum_{i\in S}a_i(x)\bigr)/|S|$ for every $x\in V(H)$.

Under the expansion assumption, we use the dual formula in
Lemma~\ref{lem:transport-duality} and the estimate in Lemma~\ref{lem:variance}
to bound the average transport cost needed in Section~\ref{sec:hitting}.

\Needspace{12\baselineskip}
\begin{lemma}\label{lem:average-transport}
Let $H$ be a connected graph on $N\geq1$ vertices, let $h\geq3$ be an
integer, and let $\ell\geq1$ be real. Suppose that every nonempty set
$W\subseteq V(H)$ with $|W|\leq N/2$ satisfies $|N_H(W)|>\frac{|W|}{\ell}$.
Let $A_1,\ldots,A_h\subseteq V(H)$ satisfy $|A_i|\geq N/h$ for every $i\in[h]$.
Then, for every integer $q$ with $1\leq q\leq h-1$, we have
\[
  \frac1{\binom{h}{q}}\sum_{\substack{S\subseteq[h]\\|S|=q}}
  D_H(a_{S^c},a_S)
  \leq\frac{\sqrt6\,h\ell}{\sqrt{q(h-q)}}.
\]
\end{lemma}

\begin{proof}
Fix $q$ with $1\leq q\leq h-1$. Let $S$ range over the subsets of $[h]$
with $|S|=q$.
Throughout this proof, sums indexed by $x$ are over $V(H)$. Define
\[
  c_{S,i}:=
  \begin{cases}
    1/q,&i\in S,\\
    -1/(h-q),&i\notin S.
  \end{cases}
\]
Each label $i$ belongs to $\binom{h-1}{q-1}$ of these sets and is absent
from $\binom{h-1}{q}$ of them. Hence
\begin{equation}\label{eq:coefficient-square-sum}
  \sum_S c_{S,i}^2
  =\frac{\binom{h-1}{q-1}}{q^2}
   +\frac{\binom{h-1}{q}}{(h-q)^2}
  =\frac{\binom{h}{q}}{q(h-q)}.
\end{equation}
Here we used $\binom{h-1}{q-1}=q\binom{h}{q}/h$ and
$\binom{h-1}{q}=(h-q)\binom{h}{q}/h$.
Since $\sum_{j=1}^{h} c_{S,j}=0$ for each $S$ and, for fixed $i$,
$\sum_S c_{S,i}c_{S,j}$ has the same value for all $j\neq i$ by
symmetry, we have
\[
  0=\sum_S c_{S,i}\sum_{j=1}^{h} c_{S,j}
   =\sum_S c_{S,i}^2+\sum_{j\neq i}\sum_S c_{S,i}c_{S,j}
   =\sum_S c_{S,i}^2+(h-1)\sum_S c_{S,i}c_{S,j}
     \qquad(j\neq i).
\]
Hence, for $j\neq i$,
\begin{equation}\label{eq:coefficient-product-sum}
  \sum_S c_{S,i}c_{S,j}
  =-\frac1{h-1}\sum_S c_{S,i}^2
  =-\frac{\binom{h}{q}}{(h-1)q(h-q)}.
\end{equation}
Since $a_S(x)-a_{S^c}(x)=\sum_{i=1}^{h} c_{S,i}a_i(x)$, we use
\eqref{eq:coefficient-square-sum} and \eqref{eq:coefficient-product-sum} to obtain
\begin{align}
  &\mathrel{\phantom{=}}\frac1{\binom{h}{q}}\sum_S\sum_x(a_S(x)-a_{S^c}(x))^2\notag\\*
  &=\frac1{\binom{h}{q}}\sum_S\sum_x
       \left(\sum_{i=1}^{h} c_{S,i}a_i(x)\right)^2\notag\\*
  &=\frac1{\binom{h}{q}}\sum_x\left[
       \sum_{i=1}^{h}\left(\sum_S c_{S,i}^2\right)a_i(x)^2
       +\sum_{i\neq j}\left(\sum_S c_{S,i}c_{S,j}\right)a_i(x)a_j(x)
     \right]\notag\\*
  &=\frac1{q(h-q)}\sum_{i=1}^{h}\sum_x a_i(x)^2
    -\frac1{(h-1)q(h-q)}\sum_{i\neq j}\sum_x a_i(x)a_j(x)\notag\\*
  &\leq\frac1{q(h-q)}\sum_{i=1}^{h}\frac1{|A_i|}\notag\\*
  &\leq\frac{h^2}{Nq(h-q)}.
  \label{eq:target-square-average}
\end{align}
The inequalities use $a_i(x)\geq0$ and
$\sum_x a_i(x)^2=1/|A_i|\leq h/N$.

\Needspace{8\baselineskip}
For any $1$-Lipschitz function $f$ on $V(H)$, put
$\overline f:=\bigl(\sum_x f(x)\bigr)/N$.
Since $\sum_x a_S(x)=\sum_x a_{S^c}(x)=1$ for each $S$, we have
$\sum_x(a_{S^c}(x)-a_S(x))=0$, and hence
\[
  \sum_x(f(x)-\overline f)(a_{S^c}(x)-a_S(x))
  =\sum_x f(x)(a_{S^c}(x)-a_S(x)).
\]
The $1$-Lipschitz function $f-\overline f$ also satisfies
$\sum_x(f(x)-\overline f)=0$. Thus, for each $S$,
Lemma~\ref{lem:transport-duality} gives a $1$-Lipschitz function $f_S$ with
$\sum_x f_S(x)=0$ such that
\[
  D_H(a_{S^c},a_S)=\sum_x f_S(x)(a_{S^c}(x)-a_S(x)).
\]
Hence Lemma~\ref{lem:variance} gives $\sum_xf_S(x)^2\leq6N\ell^2$.
Averaging over $S$ gives
\[
  \frac1{\binom{h}{q}}\sum_S\sum_x f_S(x)^2\leq6N\ell^2.
\]
Applying the Cauchy--Schwarz inequality and using~\eqref{eq:target-square-average},
we have
\begin{align*}
  \frac1{\binom{h}{q}}\sum_S D_H(a_{S^c},a_S)
  &=\frac1{\binom{h}{q}}\sum_S\sum_xf_S(x)(a_{S^c}(x)-a_S(x))\\*
  &\leq\left(\frac1{\binom{h}{q}}\sum_S\sum_x f_S(x)^2\right)^{1/2}
         \left(\frac1{\binom{h}{q}}\sum_S\sum_x
           (a_S(x)-a_{S^c}(x))^2\right)^{1/2}\\*
  &\leq\sqrt{6N}\,\ell
       \left(\frac{h^2}{Nq(h-q)}\right)^{1/2}
   =\frac{\sqrt6\,h\ell}{\sqrt{q(h-q)}},
\end{align*}
as desired.
\end{proof}

\clearpage
\section{Proof of Theorem~\ref{thm:hitting}}\label{sec:hitting}

For convenience, we restate Theorem~\ref{thm:hitting} below.

\begin{hittingrestatement}
Let $H$ be a connected graph on $N\geq1$ vertices, let $h\geq3$ be an
integer, and let $\ell\geq1$ be real. Let
$A_1,\ldots,A_h\subseteq V(H)$ satisfy $|A_i|\geq N/h$ for every
$i\in[h]$. Then at least one of the following holds:
\begin{enumerate}
\item there is a nonempty set $W\subseteq V(H)$ such that
$|W|\leq N/2$ and $|N_H(W)|\leq |W|/\ell$;
\item there is a nonempty set $T\subseteq V(H)$ such that $H[T]$ is
connected, $T\cap A_i\neq\varnothing$ for every $i\in[h]$, and
$|T|\leq11h\ell$.
\end{enumerate}
\end{hittingrestatement}

\noindent\textit{Proof overview.}
Suppose that the first conclusion of Theorem~\ref{thm:hitting} fails.
We want to choose one vertex from each $A_i$ and then order the vertices
and join them by a short walk in $H$.

Hence we construct a layered directed graph $\mathcal D$ with vertices $(S,x)$,
where $S\subsetneq[h]$ is the list of labels $i$ for which we still need
a vertex from $A_i$, and $x\in V(H)$ is the current vertex. The $q$th layer is
\[
  \mathcal L_q=\{(S,x)\in V(\mathcal D)\mid |S|=q\}.
\]
For $i\in S$ and $y\in V(H)$, we add a directed edge of length $d_H(x,y)$
from $(S,x)$ to $(S\setminus\{i\},y)$. Then seek a path from the top layer
$\mathcal L_{h-1}$ to the bottom layer $\mathcal L_0$, with each step
moving to $y$ and removing the label $i$ from the list.

We give every layer $q$ total weight $1$ by assigning weight
$1/\binom{h}{q}$ to each of its $\binom{h}{q}$ lists $S$. Since
$\sum_{x\in V(H)}a_{S^c}(x)=1$ (by the definitions on
page~\pageref{page:mass-functions}), we split this weight among the
vertices $x$ according to $a_{S^c}$, giving
\[
  \omega(S,x)=\frac{a_{S^c}(x)}{\binom{h}{q}},
  \qquad q=|S|.
\]
We note that in the top layer $\mathcal L_{h-1}$, $|S|=h-1$, so $S^c=\{i\}$
for some $i\in[h]$. If $\omega(S,x)>0$, then $a_{S^c}(x)=a_i(x)>0$, and
hence $x\in A_i$.

To construct a unit flow, we want to distribute the weight $\omega(S,x)$
of each vertex $(S,x)$ with $S\neq\varnothing$ among its outgoing edges
to the next layer. For each nonempty list $S$, we take a nonnegative
matrix $P_S$ with $\sum_{y\in V(H)}P_S(x,y)=a_{S^c}(x)$ for every
$x\in V(H)$.
For fixed $(S,x)$, we first split $\omega(S,x)$ among the vertices $y$
in proportion to $P_S(x,y)$, giving weight $P_S(x,y)/\binom{h}{q}$ to
each $y$. If $\sum_{i\in S}a_i(y)>0$, we again split this weight among
the edges $(S,x)\to(S\setminus\{i\},y)$ for $i\in S$ in proportion to
$a_i(y)$.

The weight of the edge $(S,x)\to(S\setminus\{i\},y)$ is
\[
  w_{S,i}(x,y)=
  \begin{cases}
    \displaystyle\frac{P_S(x,y)}{\binom{h}{q}}\,
    \frac{a_i(y)}{\sum_{j\in S}a_j(y)},
      &\displaystyle\sum_{j\in S}a_j(y)>0,\\[3mm]
    0,&\displaystyle\sum_{j\in S}a_j(y)=0.
  \end{cases}
\]
If we also require $\sum_{x\in V(H)}P_S(x,y)=a_S(y)$ for each nonempty
list $S$ and every $y\in V(H)$, then at each intermediate vertex of
$\mathcal D$, the total incoming weight equals the total outgoing weight.
Thus the edge weights give a unit flow from the top layer
to the bottom layer.

We note that if $w_{S,i}(x,y)>0$, then $a_i(y)>0$, and hence $y\in A_i$.
Thus, for any path $\gamma$ from a positive-weight vertex in the top
layer to the bottom layer along positive-weight edges, the second
coordinates of its vertices hit every $A_i$ in $H$.

Decompose this unit flow into paths $\gamma_k$ from $\mathcal L_{h-1}$
to $\mathcal L_0$. Let $\lambda_k$ be the amount of flow assigned to
$\gamma_k$, so that $\sum_k\lambda_k=1$. We join the corresponding
vertices in $H$ in order by shortest paths and write $L_k$ for the
length of the resulting walk.

Multiplying the length of each walk by the flow assigned to the
corresponding path and summing over all paths gives
\[
  \sum_k\lambda_k L_k
  =\sum_{q=1}^{h-1}\frac1{\binom{h}{q}}
       \sum_{|S|=q}\sum_{x,y\in V(H)}d_H(x,y)P_S(x,y).
\]
When each $P_S$ is optimal (as defined on
page~\pageref{eq:transport-definition}), Lemma~\ref{lem:average-transport}
bounds this value by $4\sqrt6\,h\ell$.
Since $\sum_k\lambda_k=1$, at least one $L_k$ is at most
$4\sqrt6\,h\ell$. Let $T$ be the set of vertices on the corresponding
walk. Then $H[T]$ is connected, $T$ meets every $A_i$, and
$|T|\leq 1+L_k<11h\ell$.

\begin{proof}[Proof of Theorem~\ref{thm:hitting}]
Suppose that the first conclusion fails. Then every nonempty
$W\subseteq V(H)$ with $|W|\leq N/2$ satisfies
$|N_H(W)|>|W|/\ell$.
For every $i\in[h]$ and every $x\in V(H)$, define
\[
  a_i(x):=
  \begin{cases}
    1/|A_i|,&x\in A_i,\\
    0,&x\notin A_i.
  \end{cases}
\]
For $S\subseteq[h]$, write $S^c:=[h]\setminus S$.
For every nonempty $S\subseteq[h]$ and every $x\in V(H)$, put
$a_S(x):=\bigl(\sum_{i\in S}a_i(x)\bigr)/|S|$.

\Needspace{7\baselineskip}
For every nonempty proper subset $S\subsetneq[h]$, choose a transport
matrix $P_S$ attaining $D_H(a_{S^c},a_S)$. Thus we have
\begin{equation}\label{eq:optimal-marginals}
  \sum_yP_S(x,y)=a_{S^c}(x),\qquad
  \sum_xP_S(x,y)=a_S(y),\qquad
  \sum_{x,y}d_H(x,y)P_S(x,y)=D_H(a_{S^c},a_S).
\end{equation}

Let $\mathcal D$ be the following finite directed graph. Its vertex set is
\[
  V(\mathcal D):=
    \{(S,x):S\subsetneq[h],\ x\in V(H)\}.
\]
Note that $S$ here may be empty.
For $q=0,\ldots,h-1$, put
$\mathcal L_q:=\{(S,x)\in V(\mathcal D):|S|=q\}$, and call $\mathcal L_q$ layer $q$.

For every nonempty $S\subsetneq[h]$, every $i\in S$, and every
$x,y\in V(H)$, add the edge
\[
  (S,x)\longrightarrow(S\setminus\{i\},y).
\]
Thus each edge goes from a layer $\mathcal L_q$ to the next layer $\mathcal L_{q-1}$.

For $(S,x)\in\mathcal L_q$, define
\[
  \omega(S,x):=\frac{a_{S^c}(x)}{\binom{h}{q}}.
\]
Give the edge $(S,x)\to(S\setminus\{i\},y)$ length $d_H(x,y)$ and weight
\begin{equation}\label{eq:layered-edge-weight}
  w_{S,i}(x,y):=
  \begin{cases}
    \displaystyle\frac{P_S(x,y)a_i(y)}{\binom{h}{q}\sum_{j\in S}a_j(y)},
       &\displaystyle\sum_{j\in S}a_j(y)>0,\\[2mm]
    0,&\displaystyle\sum_{j\in S}a_j(y)=0,
  \end{cases}
  \qquad q:=|S|.
\end{equation}

\begin{claim}\label{clm:flow-balance}
For every $(S,x)\in \mathcal L_q$ with $1\leq q\leq h-1$, its total outgoing
edge weight is $\omega(S,x)$. For every $(R,y)\in \mathcal L_{q-1}$ with
$1\leq q\leq h-1$, its total incoming edge weight is $\omega(R,y)$.
\end{claim}

\begin{proof}
Fix $(S,x)\in \mathcal L_q$. For each $y$ with $\sum_{j\in S}a_j(y)>0$,
summing~\eqref{eq:layered-edge-weight} over $i\in S$ gives
\begin{equation}\label{eq:edge-weight-sum}
  \sum_{i\in S}w_{S,i}(x,y)
  =\frac{P_S(x,y)}{\binom{h}{q}}
       \frac{\sum_{i\in S}a_i(y)}{\sum_{j\in S}a_j(y)}
  =\frac{P_S(x,y)}{\binom{h}{q}}.
\end{equation}
When $\sum_{j\in S}a_j(y)=0$, we have $a_i(y)=0$ for every $i\in S$,
and the column-sum condition in~\eqref{eq:optimal-marginals} gives
$P_S(x,y)=0$ for every $x\in V(H)$, so both the first and the last
expressions in~\eqref{eq:edge-weight-sum} are zero.
Therefore, by the row-sum condition in~\eqref{eq:optimal-marginals},
\begin{equation}\label{eq:outgoing-weight}
  \sum_y\sum_{i\in S} w_{S,i}(x,y)
  =\frac1{\binom{h}{q}}\sum_yP_S(x,y)
  =\frac{a_{S^c}(x)}{\binom{h}{q}}
  =\omega(S,x).
\end{equation}

For a fixed $i\in S$ and a vertex $y$ with $\sum_{j\in S}a_j(y)>0$,
the column-sum condition gives
\[
  \begin{aligned}
    \sum_xw_{S,i}(x,y)
    &=\frac{a_i(y)}{\binom{h}{q}\sum_{j\in S}a_j(y)}
      \sum_xP_S(x,y)\\
    &=\frac{a_i(y)}{\binom{h}{q}\sum_{j\in S}a_j(y)}
      \cdot\frac{\sum_{j\in S}a_j(y)}q\\
    &=\frac{a_i(y)}{q\binom{h}{q}}.
  \end{aligned}
\]
When $\sum_{j\in S}a_j(y)=0$, both the first and the last expressions
are zero.

Now fix $R\subseteq[h]$ with $|R|=q-1$.
An edge entering $(R,y)$ must come from
$(R\cup\{i\},x)$ for some $i\notin R$. These are all possible
incoming edges, so, by the preceding equality,
\begin{align}
  \sum_{i\notin R}\sum_xw_{R\cup\{i\},i}(x,y)
  &=\frac1{q\binom{h}{q}}\sum_{i\notin R}a_i(y)\notag\\*
  &=\frac{h-q+1}{q\binom{h}{q}}\,a_{R^c}(y)\notag\\*
  &=\frac{a_{R^c}(y)}{\binom{h}{q-1}}
   =\omega(R,y),\label{eq:incoming-weight}
\end{align}
where we used
$q\binom{h}{q}=(h-q+1)\binom{h}{q-1}$.
This proves both assertions.
\end{proof}

By Claim~\ref{clm:flow-balance}, the total incoming and outgoing
weights are equal at every vertex in $\mathcal L_1\cup\cdots\cup \mathcal L_{h-2}$.
Moreover, the total outgoing weight from $\mathcal L_{h-1}$ is
\begin{equation}\label{eq:total-top-weight}
  \sum_{|S|=h-1}\sum_x\omega(S,x)=\frac1h\sum_{|S|=h-1}\sum_xa_{S^c}(x)=\frac1h\sum_{i=1}^h\sum_xa_i(x)=1.
\end{equation}

\begin{claim}\label{clm:path-decomposition}
There are directed paths $\gamma_1,\ldots,\gamma_m$ in $\mathcal D$,
each starting in $\mathcal L_{h-1}$ and ending in $\mathcal L_0$, and positive numbers
$\lambda_1,\ldots,\lambda_m$ such that
\begin{equation}\label{eq:path-decomposition}
  w(e)=\sum_{\substack{k\in[m]\\e\in E(\gamma_k)}}\lambda_k
  \quad(e\in E(\mathcal D)),\qquad
  \sum_{k=1}^m\lambda_k=1.
\end{equation}
Here $w(e)$ denotes the weight assigned in~\eqref{eq:layered-edge-weight}.
\end{claim}

\begin{proof}
Give every edge $e$ an initial weight equal to $w(e)$.
By Claim~\ref{clm:flow-balance}, the total residual weight of the incoming
edges initially equals that of the outgoing edges at every intermediate
vertex. Whenever the total outgoing residual weight from $\mathcal L_{h-1}$ is
positive, we can therefore find a directed path $\gamma$ from $\mathcal L_{h-1}$
to $\mathcal L_0$ all of whose edges have positive residual weight.

Let $\lambda>0$ be the smallest residual weight of an edge of
$\gamma$. Record the pair $(\gamma,\lambda)$, and subtract $\lambda$
from the residual weight of every edge of $\gamma$.
All residual weights remain nonnegative, and at least one previously
positive edge now has residual weight zero.
At each intermediate vertex of $\gamma$, exactly one incoming and
one outgoing edge have been decreased by $\lambda$. Thus their total
residual weights remain equal. At every other intermediate vertex,
the equality is unchanged.

Repeat this operation whenever a vertex of $\mathcal L_{h-1}$ has positive
outgoing residual weight. No residual edge weight increases, and
each operation makes at least one positive edge zero. Since
$\mathcal D$ has finitely many edges, only finitely many operations
are possible.
At termination, every edge has zero residual weight, by conservation
and nonnegativity applied successively down the layers.

Denote the recorded pairs by $(\gamma_k,\lambda_k)$, $k\in[m]$.
Since an edge loses precisely $\lambda_k$ whenever it belongs to
$\gamma_k$, we have proved the first equality in
\eqref{eq:path-decomposition}. Each operation decreases the total
outgoing residual weight from $\mathcal L_{h-1}$ by $\lambda_k$.
That total is initially one by~\eqref{eq:total-top-weight} and finally
zero, so $\sum_k\lambda_k=1$.
\end{proof}

Consider one of these paths, and write its successive vertices as
\[
  (S_1,x_1),(S_2,x_2),\ldots,(S_h,x_h),
  \qquad |S_j|=h-j.
\]
There is a unique label $i_1$ such that
$S_1=[h]\setminus\{i_1\}$. For $j=1,\ldots,h-1$, let $i_{j+1}$ be the
label removed by the edge from $(S_j,x_j)$ to $(S_{j+1},x_{j+1})$.
Then
$S_{j+1}=S_j\setminus\{i_{j+1}\}$, and
$i_1,\ldots,i_h$ is a permutation of $[h]$.

\Needspace{5\baselineskip}
Every edge of the path has positive original weight.
In particular, its first vertex has positive outgoing weight, so,
by the definition of $\omega$ and Claim~\ref{clm:flow-balance}, we obtain
$0<\omega(S_1,x_1)=a_{i_1}(x_1)/h$.
Hence $x_1\in A_{i_1}$.
For $j=1,\ldots,h-1$, positivity of
$w_{S_j,i_{j+1}}(x_j,x_{j+1})$ in~\eqref{eq:layered-edge-weight}
implies $a_{i_{j+1}}(x_{j+1})>0$, and hence
$x_{j+1}\in A_{i_{j+1}}$.
Thus the sequence $x_1,\ldots,x_h$ contains a vertex from every $A_i$.
These vertices need not be distinct.

For $k\in[m]$, let
$L_k:=\sum_{e\in E(\gamma_k)}\operatorname{len}(e)$, where
$\operatorname{len}(e)=d_H(x,y)$ for
$e=(S,x)\to(S\setminus\{i\},y)$.
Hence
\begin{align}
  \sum_{k=1}^m\lambda_k L_k
  &=\sum_{k=1}^m\lambda_k
       \sum_{e\in E(\gamma_k)}\operatorname{len}(e)\notag\\*
  &=\sum_{e\in E(\mathcal D)}\operatorname{len}(e)
       \sum_{\substack{k\in[m]\\e\in E(\gamma_k)}}\lambda_k\notag\\*
  &=\sum_{e\in E(\mathcal D)}\operatorname{len}(e)w(e)\notag\\*
  &=\sum_{q=1}^{h-1}\sum_{|S|=q}
       \sum_{i\in S}\sum_{x,y}d_H(x,y)w_{S,i}(x,y)\notag\\*
  &=\sum_{q=1}^{h-1}\frac1{\binom{h}{q}}
       \sum_{|S|=q}\sum_{x,y}d_H(x,y)P_S(x,y)\notag\\*
  &=\sum_{q=1}^{h-1}\frac1{\binom{h}{q}}
       \sum_{|S|=q}D_H(a_{S^c},a_S)\notag\\*
  &\leq\sqrt6\,h\ell
       \sum_{q=1}^{h-1}\frac1{\sqrt{q(h-q)}},
  \label{eq:total-transport-length}
\end{align}
where the third equality uses~\eqref{eq:path-decomposition}, the fifth uses
\eqref{eq:edge-weight-sum}, and the sixth uses~\eqref{eq:optimal-marginals}.
The inequality uses Lemma~\ref{lem:average-transport}.

The summands in the last expression are unchanged when $q$ is
replaced by $h-q$. By symmetry, we have
\begin{align*}
  \sum_{q=1}^{h-1}\frac1{\sqrt{q(h-q)}}
  &\leq2\sum_{q=1}^{\lfloor h/2\rfloor}
              \frac1{\sqrt{q(h-q)}}\\*
  &\leq2\sqrt{\frac2h}
              \sum_{q=1}^{\lfloor h/2\rfloor}\frac1{\sqrt q}\\*
  &\leq4\sqrt{\frac{2\lfloor h/2\rfloor}{h}}\leq4.
\end{align*}
The second inequality uses $h-q\geq h/2$ in this range.
For the third, sum
$1/\sqrt q\leq2(\sqrt q-\sqrt{q-1})$ from
$q=1$ to $\lfloor h/2\rfloor$.
It follows from~\eqref{eq:total-transport-length} that
$\sum_k\lambda_kL_k\leq4\sqrt6\,h\ell$.
Since all $\lambda_k$ are positive and $\sum_k\lambda_k=1$,
at least one $L_k$ is at most $4\sqrt6\,h\ell$.

Choose such a path and let $x_1,\ldots,x_h$ be its second coordinates.
For every $j\in[h-1]$, choose a shortest path in $H$ from $x_j$ to
$x_{j+1}$. Let $T$ be the union of the vertex sets of these paths.
Clearly, $H[T]$ is connected. The set $T$ contains every $x_j$, and hence meets every
$A_i$.
Starting with $x_1$, the $j$th path adds at most $d_H(x_j,x_{j+1})$
vertices. Consequently,
\[
  |T|\leq1+\sum_{j=1}^{h-1}d_H(x_j,x_{j+1})
       \leq1+4\sqrt6\,h\ell<11h\ell.
\]
Here $h\ell\geq1$ and $1+4\sqrt6<11$.
Thus $T$ satisfies the second conclusion.
\end{proof}

\clearpage
\section{Proof of the main theorem}\label{sec:separator}

The proof follows the separator construction of
Le~\cite[Section~3.1]{le} (going back to Plotkin, Rao, and
Smith~\cite[Section~2]{plotkin-rao-smith} and Alon, Seymour and
Thomas~\cite{alon-seymour-thomas}).

For convenience, we restate Theorem~\ref{thm:main} below.

\begin{mainrestatement}
Let $h\geq3$ be an integer, and let $G$ be a $K_h$-minor-free graph on
$n$ vertices. Then $G$ has a balanced separator $S$ such that
\[
  |S|\leq14h\sqrt n.
\]
\end{mainrestatement}

\begin{proof}[Proof of Theorem~\ref{thm:main}]
If $0\leq n\leq h^2$, take $S:=V(G)$. Then
$|S|=n\leq h\sqrt n$, and $G-S$ has no components.
We may therefore assume that $n>h^2$. Put
$\ell:=\frac{\sqrt n}{h}>1$, so $n=h^2\ell^2$.
If every component of $G$ has at most $2n/3$ vertices, take
$S:=\varnothing$. Otherwise $G$ has a unique component of order
greater than $2n/3$.

We construct a sequence of triples $(H,X,\mathcal K)$.
In each triple, $H$ is a connected induced subgraph of $G$ with
$v(H)>2n/3$, $X$ is a subset of $V(G)\setminus V(H)$, and
$\mathcal K$ is a family of vertex sets satisfying the following
conditions:
\begin{enumerate}
\item[(i)] The sets in $\mathcal K$ are nonempty, pairwise disjoint
subsets of $V(G)\setminus(V(H)\cup X)$. For every $C\in\mathcal K$,
the graph $G[C]$ is connected and $N_H(C)\neq\varnothing$.
Moreover, $e_G(C,C')>0$ for all distinct $C,C'\in\mathcal K$.
\item[(ii)] If $uv\in E(G)$, $u\in V(H)$, and $v\notin V(H)$, then
$v\in X\cup\bigcup_{C\in\mathcal K}C$.
\item[(iii)] For every $C\in\mathcal K$,
$\min\{|C|,|N_H(C)|\}\leq13h\ell$.
\end{enumerate}
Initially, let $H$ be the component of $G$ of order greater than
$2n/3$, and take $X:=\varnothing$ and $\mathcal K:=\varnothing$.
There is no edge from this $H$ to another component of $G$, so
condition (ii) holds. Conditions (i) and (iii) have no sets $C$ to
check, and hence also hold.

If $\mathcal K\neq\varnothing$, then the sets in $\mathcal K$ form the
branch sets of a complete minor.
Since $G$ is $K_h$-minor-free, condition (i) implies
$|\mathcal K|\leq h-1$.

We use the following three modifications of a triple.
In each case, $R$ denotes the induced subgraph remaining from $H$,
and $X^+$ and $\mathcal K^+$ denote the updated set and family.
\begin{enumerate}
\item[(a)] Suppose there is a nonempty set $T\subseteq V(H)$ such that
$H[T]$ is connected, $|T|\leq13h\ell$, and
$T\cap N_H(C)\neq\varnothing$ for every $C\in\mathcal K$.
Put
\[
  R:=H-T,\qquad X^+:=X,\qquad
  \mathcal K^+:=\mathcal K\cup\{T\}.
\]
The last condition on $T$ ensures that it is adjacent to every
existing branch set.

\item[(b)] Suppose there are $C\in\mathcal K$ and a nonempty set
$Z\subseteq V(H)$ such that $|Z|\leq v(H)/2$, $G[C\cup Z]$ is connected,
and $|N_H(C\cup Z)|\leq13h\ell$. Put
\[
  R:=H-Z,\qquad X^+:=X,\qquad
  \mathcal K^+:=(\mathcal K\setminus\{C\})\cup\{C\cup Z\}.
\]

\item[(c)] Suppose a nonempty set $W\subseteq V(H)$ satisfies
$|W|\leq v(H)/2$ and $|N_H(W)|\leq |W|/\ell$. Put
\[
  R:=H-(W\cup N_H(W)),\qquad
  X^+:=X\cup N_H(W),\qquad \mathcal K^+:=\mathcal K.
\]
\end{enumerate}
If every component of $R$ has at most $2n/3$ vertices, we stop the construction.
Otherwise let $H'$ be the unique component of $R$ of order greater
than $2n/3$, and put
\[
  X':=X^+,\qquad
  \mathcal K':=\{C\in\mathcal K^+:N_{H'}(C)\neq\varnothing\}.
\]
The next triple is $(H',X',\mathcal K')$.

\begin{claim}\label{clm:preservation}
If a modification produces a component $H'$ of order greater than
$2n/3$, then $(H',X',\mathcal K')$ satisfies (i)--(iii).
\end{claim}

\begin{proof}
We first check that each of the modifications (a)--(c) preserves (i).
In (a), $T$ is disjoint from the old branch sets and from $X$, since
$T\subseteq V(H)$. It is connected and adjacent to every old branch
set by its defining conditions.
In (b), the same disjointness holds for $Z$, the set $C\cup Z$ is
connected, and it remains adjacent to every other branch set because
it contains $C$.
In (c), the added vertices of $X'$ belong to $V(H)$, so they are
disjoint from every branch set.
Each modification removes from $H$ every vertex it adds to $X$ or to a
branch set, so $X'$ and every member of $\mathcal K^+$ lie outside $H'$.
Deleting members of $\mathcal K^+$ does not affect disjointness,
connectivity, or pairwise adjacency of the remaining members.
Every member of $\mathcal K'$ has a neighbor in $H'$ by definition.

We next check that each of the modifications (a)--(c) preserves (ii).
In (a) and (b), the vertices removed from $H$ belong to a member
of $\mathcal K^+$. In (c), every removed vertex having a neighbor
in $R$ belongs to $N_H(W)\subseteq X^+$.
By the old condition (ii), every vertex outside $R$ with a neighbor
in $R$ therefore belongs to $X^+\cup\bigcup_{C\in\mathcal K^+}C$.
This property is preserved when passing to the component $H'$ of $R$
and discarding members of $\mathcal K^+$ with no neighbor in $H'$.

Finally, we check that each of the modifications (a)--(c) preserves (iii).
For the new branch set in (a), the required bound follows from
$|T|\leq13h\ell$.
For the enlarged branch set in (b),
$N_{H'}(C\cup Z)\subseteq N_H(C\cup Z)$,
so (b) gives
\[
  \min\{|C\cup Z|,|N_{H'}(C\cup Z)|\}
  \leq |N_H(C\cup Z)|\leq13h\ell.
\]

For every unchanged branch set $C$, we have
$N_{H'}(C)\subseteq N_H(C)$. Hence, by the old
condition (iii),
\[
  \min\{|C|,|N_{H'}(C)|\}
  \leq\min\{|C|,|N_H(C)|\}
  \leq13h\ell.\qedhere
\]
\end{proof}

\begin{claim}\label{clm:available-modification}
Every triple $(H,X,\mathcal K)$ satisfying (i)--(iii) and
$v(H)>2n/3$ admits at least one modification of type (a), (b), or (c).
\end{claim}

\begin{proof}
If $\mathcal K=\varnothing$, choose any $v\in V(H)$ and put
$T:=\{v\}$. Then $H[T]$ is connected and
$|T|=1\leq13h\ell$. There are no branch sets whose neighborhoods
$T$ must meet, so (a) applies.
We may therefore assume that $\mathcal K\neq\varnothing$.
If the sets in (b) exist, use (b). Henceforth assume that no
modification of type (b) is possible.

Fix $C\in\mathcal K$. For every integer $r\geq0$, put
\[
  B_r(C):=\{v\in V(H):d_H(v,N_H(C))\leq r\}.
\]
The set $N_H(C)$ is nonempty by (i), so every $B_r(C)$ is nonempty.
By (i) and the definition of $B_r(C)$, the graph $G[C\cup B_r(C)]$
is connected.

Since $N_H(C)\subseteq B_r(C)$, no vertex of
$V(H)\setminus B_r(C)$ has a neighbor in $C$.
Clearly,
\[
  N_H(C\cup B_r(C))=B_{r+1}(C)\setminus B_r(C).
\]
Suppose that $|B_r(C)|\leq v(H)/2$ for every integer $r$ with
$0\leq r<\lceil\ell\rceil$. For each such $r$, the set $Z:=B_r(C)$
satisfies all the conditions in (b) except possibly the bound on
$|N_H(C\cup Z)|$. Since (b) is impossible, that bound must fail. Hence
\[
  |B_{r+1}(C)\setminus B_r(C)|
  >13h\ell.
\]
Since the sets $B_{r+1}(C)\setminus B_r(C)$ are pairwise disjoint,
\[
  |B_{\lceil\ell\rceil}(C)|
  \geq\sum_{r=0}^{\lceil\ell\rceil-1}
           |B_{r+1}(C)\setminus B_r(C)|
  >13h\ell\lceil\ell\rceil\geq13h\ell^2.
\]
Otherwise, $|B_r(C)|>v(H)/2$ for some $r<\lceil\ell\rceil$. The
sets are nested, so $|B_{\lceil\ell\rceil}(C)|>v(H)/2$.
In both cases,
\begin{equation}\label{eq:large-neighborhood}
  |B_{\lceil\ell\rceil}(C)|
  \geq\min\{v(H)/2,13h\ell^2\}
  =\min\{v(H)/2,13n/h\}\geq v(H)/h.
\end{equation}
The equality uses $n=h^2\ell^2$. The last inequality follows
from $h\geq3$ and $v(H)\leq n$.

Write $\mathcal K=\{C_1,\ldots,C_s\}$, where
$1\leq s=|\mathcal K|\leq h-1$.
For $i\in[s]$, put $A_i:=B_{\lceil\ell\rceil}(C_i)$; for
$s<i\leq h$, put $A_i:=A_1$.
By~\eqref{eq:large-neighborhood}, every target set has size at least
$v(H)/h$. Apply Theorem~\ref{thm:hitting} to $H$ and these $h$ sets.
If its first conclusion holds, then (c) applies.

Otherwise there is a set $T_0\subseteq V(H)$ such that $H[T_0]$ is
connected, $|T_0|\leq11h\ell$, and $T_0$ meets every $A_i$.
For each $i\in[s]$, choose
$v_i\in T_0\cap B_{\lceil\ell\rceil}(C_i)$ and a shortest path $P_i$
in $H$ from $v_i$ to $N_H(C_i)$.
The definition of $B_{\lceil\ell\rceil}(C_i)$ gives
$|E(P_i)|\leq\lceil\ell\rceil\leq2\ell$, since $\ell>1$.
Put
\[
  T:=T_0\cup\bigcup_{i=1}^sV(P_i).
\]
Every $P_i$ starts in $T_0$, so $H[T]$ is connected.
Its other endpoint belongs to $N_H(C_i)$, so $T$ meets every
branch-set neighborhood. Since the starting vertex of $P_i$ is
already in $T_0$, the path adds at most $|E(P_i)|$ vertices.
The required bound for (a) follows:
\[
  |T|\leq |T_0|+\sum_{i=1}^s|E(P_i)|
  \leq11h\ell+2s\ell\leq13h\ell.\qedhere
\]
\end{proof}

Starting from the initial triple, choose an available modification
given by Claim~\ref{clm:available-modification}.
If it does not stop the construction, use the next triple
$(H',X',\mathcal K')$.
Claims~\ref{clm:preservation} and~\ref{clm:available-modification}
permit the argument to be repeated.
Since each modification removes a nonempty set $T$, $Z$, or $W$
from $H$, the next graph $H'$ has strictly fewer vertices than $H$.
There can be at most $n$ such decreases. Since a modification is
available at every triple with $v(H)>2n/3$, the construction must
stop after finitely many modifications.

Let $(H_0,X_0,\mathcal K_0)$ be the triple immediately before the
stopping modification.

Let $H_j$ be the current graph and $W_j$ the chosen set in the
$j$th modification of type (c) before the stopping modification.
Each such modification removes $W_j$ from the current graph.
The successive subgraphs are nested, so the $W_j$ are pairwise
disjoint. Since $X$ is initially empty and grows only in (c), we have
\begin{equation}\label{eq:X-bound}
  |X_0|\leq\sum_j |N_{H_j}(W_j)|
  \leq\frac{1}{\ell}\sum_j |W_j|
  \leq\frac{n}{\ell}.
\end{equation}

For each $C\in\mathcal K_0$, put $F_C:=C$ if
$|C|\leq|N_{H_0}(C)|$, and put $F_C:=N_{H_0}(C)$ otherwise.
Let $F_0:=\bigcup_{C\in\mathcal K_0}F_C$.
Thus $|F_C|=\min\{|C|,|N_{H_0}(C)|\}$.
By condition (iii) and the bound $|\mathcal K_0|\leq h-1$,
\begin{equation}
  |F_0|
  \leq\sum_{C\in\mathcal K_0}|F_C|
  \leq13h\ell|\mathcal K_0|
  \leq13h(h-1)\ell.
  \label{eq:F-bound}
\end{equation}

Let $T$, $Z$, or $W$ denote the set chosen in that modification.
Define
\begin{equation}\label{eq:final-separator}
  S:=
  \begin{cases}
    X_0\cup F_0\cup T,&\text{in case (a)},\\
    X_0\cup F_0\cup N_{H_0}(C\cup Z),&\text{in case (b)},\\
    X_0\cup F_0\cup N_{H_0}(W),&\text{in case (c)}.
  \end{cases}
\end{equation}
We verify that $S$ is balanced.
By (ii), deleting $X_0\cup F_0$ separates the remaining vertices of
$H_0$ from all remaining vertices outside $H_0$. Since
$X_0\cup F_0\subseteq S$, every component of $G-S$ lies entirely
inside $H_0$ or entirely outside $H_0$.
Since $v(H_0)>2n/3$, there are fewer than $n/3$ vertices outside $H_0$,
so every component outside $H_0$ has fewer than $n/3$ vertices.

In case (a), the graph left from $H_0$ at the stopping modification is
$H_0-T$. By the stopping rule, every component of $H_0-T$ has at most
$2n/3$ vertices. The same bound holds for every component of $G-S$
inside $H_0$.

In case (b), every component of $H_0-N_{H_0}(C\cup Z)$ is contained either in
$Z$ or in $V(H_0)\setminus Z$. A component inside $Z$ has at most
$|Z|\leq v(H_0)/2\leq n/2$ vertices.
A component outside $Z$ is contained in a component of
$H_0-Z$, all of which have at most $2n/3$ vertices by the stopping rule.
Further deletion of $X_0\cup F_0$ cannot increase any component size.

Case (c) is analogous to case (b).
Thus $S$ is balanced in all three cases.

In cases (a) and (b) of~\eqref{eq:final-separator}, $T$ and
$N_{H_0}(C\cup Z)$, respectively, have size at most $13h\ell$.
By~\eqref{eq:X-bound} and~\eqref{eq:F-bound},
\[
  |S|\leq\frac n\ell+13h(h-1)\ell+13h\ell
       =13h^2\ell+\frac n\ell.
\]

In case (c), since the sets $W_j$ were removed in earlier modifications,
$W\subseteq V(H_0)$ is disjoint from every $W_j$.
By the bound in (c),
\[
  |X_0\cup N_{H_0}(W)|
  \leq\frac{\sum_j|W_j|+|W|}{\ell}
  \leq\frac n\ell.
\]
Together with~\eqref{eq:F-bound}, this gives
$|S|\leq13h(h-1)\ell+n/\ell$.

Thus in all three cases, using $\ell=\sqrt n/h$, we have
\[
  |S|\leq13h^2\ell+\frac n\ell=14h\sqrt n.
\]
This proves the theorem.
\end{proof}

\clearpage
\section{Acknowledgments}

The author began working on the separator conjecture in mid-September 2026.
The first attempt used eight GPT 6.1 Sol (Max) sessions working in the
same Markdown file but failed.

The author then used ChatGPT to look up the $O(h\sqrt n+f(h))$
separator bound of Kawarabayashi and Reed~\cite{kawarabayashi-reed}.
The author asked GPT 6.1 Sol (Ultra) to apply the method of
Kawarabayashi and Reed~\cite{kawarabayashi-reed} to reduce $f(h)$ using
the polynomial graph minor structure theorem of Gorsky, Seweryn and
Wiederrecht~\cite{gorsky-seweryn-wiederrecht}. The model produced a
draft argument for $f(h)=O(h^{650})$. If correct, the argument would
answer Le's question~\cite[Section~1]{le} about whether $f(h)$ can be
bounded by a polynomial in $h$ using the polynomial graph minor
structure theorem.
Further progress on this approach proved difficult, and the author gave
it up.

After ChatGPT 6pro gained the ability to call four agents, the author
had 50 free messages per week with this model.
The author returned to the separator conjecture and
asked ChatGPT to improve the bound $O(h\sqrt{n\log h})$ from Le's May
preprint~\cite[Theorem~2]{le}. After two rounds, ChatGPT produced a
draft argument for the smaller bound $O(h\sqrt{n\log\log h})$.
Further attempts to improve this bound stalled.

While reading the draft proof, the author found a statement similar to
Theorem~\ref{thm:hitting} with $|T|=O(h\ell\log\log h)$. The author
thought that the $\log\log h$ factor could be removed. A search with
ChatGPT found no counterexample. The author then asked ChatGPT 6pro
to prove the stronger statement. The model gave a proof with
$|T|=O(h\ell)$. In the separator construction, this gave the
$O(h\sqrt n)$ bound in Conjecture~\ref{conj:ast}.

The author then refined the idea and rewrote the proof, using Codex
to assist with writing (editing formulas, polishing the language, and
typesetting). The author takes primary responsibility for the paper.
The author subsequently posted the paper as
\href{https://arxiv.org/abs/2610.40270v1}{arXiv:2610.40270v1}.

A few days after posting v1, the author learned that Conroy, Le and
Spalding-Jamieson~\cite{conroy-le-spalding-jamieson} had independently
obtained the $O(h\sqrt n)$ separator bound. The author sincerely thanks Professor Hung Le
for his suggestions on v1.

\end{document}